\documentclass[11pt, a4paper]{article}
\usepackage[T1]{fontenc}
\usepackage[utf8]{inputenc}
\usepackage{authblk}
\usepackage{amsmath,amsfonts,amssymb,amscd,amsthm,xspace}
\usepackage{algorithm,algpseudocode}
\algrenewcommand\algorithmicrequire{\textbf{Input:}}
\algrenewcommand\algorithmicensure{\textbf{Output:}}
\newcommand{\HRule}{\rule{\linewidth}{0.25mm}}
\newtheorem{theorem}{Theorem}
\usepackage{graphicx}
\usepackage{hyperref}
\usepackage{vmargin}
\setmarginsrb  { 0.8in}  
               { 0.6in}  
               { 0.8in}  
               { 0.8in}  
               {  20pt}  
               {0.25in}  
               {   9pt}  
               { 0.3in}  

\title{Symbolic Algorithm for Solving SLAEs with Heptadiagonal
Coefficient Matrices}
\author[1]{Milena Veneva \thanks{milena.p.veneva@gmail.com}}
\author[1]{Alexander Ayriyan \thanks{ayriyan@jinr.ru}}
\affil[1]{Joint Institute for Nuclear Research, Laboratory of Information 
Technologies, Joliot-Curie 6, 141980 Dubna, Moscow region, Russia}

\date{}
\begin{document}
\maketitle
\abstract{%
This paper presents a symbolic algorithm for solving band matrix
systems of linear algebraic equations with heptadiagonal coefficient matrices.
The algorithm is given in pseudocode. A theorem which gives the condition for
the algorithm to be stable is formulated and proven.
}
\maketitle
%

\section{Introduction}

Systems of linear algebraic equations~(SLAEs) with heptadiagonal coefficient
matrices may arise after many different scientific and engineering problems,
as well as problems of the computational linear algebra where finding the
solution of a SLAE is considered to be one of the most important problems.
For instance, a semi-implicit formulation for the discretization of the
transient terms of the system of partial differential equations~(PDEs) which
models a multiphase fluid flow in porous media yields to a heptadiagonal system
of pressure equations for each time step~(see~\cite{Kim}). On the other hand,
in~\cite{Duran} the 3D problem, simulating the incompressible blood flow in
arteries with a structured mesh domain leads to a heptadiagonal SLAE.

A whole branch of symbolic algorithms for solving systems of linear algebraic
equations with different coefficient matrices exists in the literature.
\cite{El-Mikkawy}~considers a tridiagonal matrix and a symbolic version of the
Thomas method~\cite{Higham} is formulated. The authors of~\cite{Karawia_2013a}
build an algorithm in the case of a general bordered tridiagonal SLAE, while
in~\cite{Atlan} the coefficient matrix taken into consideration is a general
opposite-bordered tridiagonal one. A pentadiagonal coefficient matrix is of
interest in~\cite{Askar}, while a cyclic pentadiagonal coefficient matrix is
considered in~\cite{Jia}. The latter algorithm can be applied to periodic
tridiagonal and periodic pentadiagonal SLAE either by setting the corresponding
matrix terms to be zero.

A performance analysis of effective methods (both numerical and symbolic) for
solving band matrix SLAEs (with three and five diagonals) being implemented in
\texttt{C++} and run on modern (as of 2018) computer systems is made by us
in~\cite{Veneva_2019}. Different strategies (symbolic included) for solving
band matrix SLAEs (with three and five diagonals) are explored by us
in~\cite{Veneva_2018}. A performance analysis of effective symbolic algorithms
for solving band matrix SLAEs with coefficient matrices with three, five and
seven diagonals being implemented in both \texttt{C++} and \texttt{Python} and
run on modern (as of 2018) computer systems is made by us in~\cite{Veneva}.
Note that the algorithm for solving a SLAE with a heptadiagonal coefficient
matrix considered in~\cite{Veneva} is the one that is going to be introduced
in the next Section.

After obtaining the algorithm independently, it has been found in the
article~\cite{Karawia_2013b} where it is applied for cyclic heptadiagonal SLAEs.
Thus, we do not claim out priority to this algorithm. However, the
novelties of this work are as follows: pure heptadiagonality, proved necessary
requirements, classical Thomas expressions i.\,e. the algorithm's formalism
follows the form of expressions that are usually used in the Thomas
algorithm for tridiagonal SLAE~\cite{Higham}, that is, the solution is searched
in the form: $y_{i} = \alpha_{i+1}y_{i+1} + \beta_{i+1}$.


\section{The Algorithm}

Let us consider a SLAE $Ax = y$, where $A$ is a $N\times N$ heptadiagonal matrix,\\
$A=\textrm{heptadiag}(\mathbf{c^{*}}, \mathbf{b^{*}}, \mathbf{a^{*}}, \mathbf{d},
\mathbf{a}, \mathbf{b}, \mathbf{c})$, $x$ and $y$ are vectors of length $N$:
\begin{equation*}
\begin{pmatrix}
d_{0} & a_{0} & b_{0} & c_{0} & 0 & \dots  & \dots & \dots & \dots & 0 \\
a_{1}^{*} & d_{1} & a_{1} & b_{1} & c_{1} & 0 & \dots & \dots & \dots & 0 \\
b_{2}^{\star} & a_{2}^{\star} & d_{2} & a_{2} & b_{2} & c_{2} & 0 & \dots & \dots & 0 \\
c_{3}^{\star} & b_{3}^{\star} & a_{3}^{\star} & d_{3} & a_{3} & b_{3} & c_{3} & 0 & \dots & 0 \\
0 & c_{4}^{\star} & b_{4}^{\star} & a_{4}^{\star} & d_{4} & a_{4} & b_{4} & c_{4} & \ddots & 0 \\
\vdots & \ddots & \ddots & \ddots & \ddots & \ddots & \ddots & \ddots & \ddots & \vdots \\
0 & \dots & 0 & c_{N-4}^{\star} & b_{N-4}^{\star} & a_{N-4}^{\star} & d_{N-4} & a_{N-4} & b_{N-4} & c_{N-4} \\
0 & \dots & \dots & 0 & c_{N-3}^{\star} & b_{N-3}^{\star} & a_{N-3}^{\star} & d_{N-3} & a_{N-3} & b_{N-3} \\
0 & \dots & \dots & \dots & 0 & c_{N-2}^{\star} & b_{N-2}^{\star} & a_{N-2}^{\star} & d_{N-2} & a_{N-2} \\
0 & \dots & \dots & \dots & \dots & 0 & c_{N-1}^{\star} & b_{N-1}^{\star} & a_{N-1}^{\star} & d_{N-1}
\end{pmatrix}
\begin{pmatrix}
x_{0} \\
x_{1} \\
x_{2} \\
x_{3} \\
x_{4} \\
\vdots \\
x_{N-4} \\
x_{N-3} \\
x_{N-2} \\
x_{N-1}
\end{pmatrix}
=
\begin{pmatrix}
y_{0} \\
y_{1} \\
y_{2} \\
y_{3} \\
y_{4} \\
\vdots \\
y_{N-4} \\
y_{N-3} \\
y_{N-2} \\
y_{N-1}
\end{pmatrix}.
\end{equation*}

A symbolic algorithm for solving SLAEs with a heptadiagonal coefficient matrix
is considered. It is based on LU decomposition in which the system $Ax = y$ is
rewritten as $LUx = y$, where $L$ and $U$ are a lower triangular and an upper
triangular matrices, respectively. The algorithm consists of two steps -- the
$LU$ decomposition together with the downwards sweep $Lz = y$ happen during the
first step, leading us from $Ax = y$ to $Ux = z$, while the upwards sweep
(solving $Ux = z$ for $x$) is done during the second step.
%

\textbf{Remark:} seems that the expression for $g_{i-2}$ is missing in the
\texttt{for} loop on p.~435 of~\cite{Karawia_2013b}. Also, the expression for
$k_{1}$ on p.~435 of the same paper is not used anywhere, so it is probably
a leftover from a previous algorithm.

Now we shall formulate a symbolic algorithm for solving such a SLAE.
\\\HRule\\[-0.1cm]
\textbf{Algorithm 1.} Symbolic algorithm for solving a SLAE $Ax = y$.
\\[-0.1cm]
\HRule\\[-0.1cm]
\begin{algorithmic}[1]
\label{alg:1}
\Require{$\mathbf{c^{*}}, \mathbf{b^{*}}, \mathbf{a^{*}}, \mathbf{d},
\mathbf{a}, \mathbf{b}, \mathbf{c}, \mathbf{y}, \varepsilon$}
\Ensure{$\mathbf{x}$}
\If{$det(A)==0$}\State{Exit.}\EndIf
\State{bool flag = False}
\State{$\mu_{0}:=d_{0}$} \Comment{Step 1.(0)}
\If{$|\mu_{0}|<\varepsilon$} \State{$\mu_{0}:=$ symb; flag = True}\EndIf
\State{$\alpha_{0}:=\dfrac{a_{0}}{\mu_{0}};\quad\beta_{0}:=\dfrac{b_{0}}{\mu_{0}};\quad \gamma_{0}:=\dfrac{c_{0}}{\mu_{0}};\quad\delta_{0}:=0.0$}
\State{$\delta_{1}:=a_{1}^{*};\quad\mu_{1}:=d_{1}-\alpha_{0}\delta_{1}$} \Comment{(1)}
\If{!flag}\If{$|\mu_{1}|<\varepsilon$}\State{$\mu_{1}:=$ symb; flag = True}\EndIf\EndIf
\State{$\alpha_{1}:=\dfrac{a_{1}-\beta_{0}\delta_{1}}{\mu_{1}};\quad
\beta_{1}:=\dfrac{b_{1}-\gamma_{0}\delta_{1}}{\mu_{1}};\quad
\gamma_{1}:=\dfrac{c_{1}}{\mu_{1}}$}
\State{$\delta_{2}:=a_{2}^{*}-\alpha_{0}b_{2}^{*};\quad
\mu_{2}:=d_{2}-\alpha_{1}\delta_{2}-\beta_{0}b_{2}^{*}$} \Comment{(2)}
\If{!flag}\If{$|\mu_{2}|<\varepsilon$}\State{$\mu_{2}:=$ symb; flag = True}\EndIf\EndIf
\State{$\alpha_{2}:=\dfrac{a_{2}-\beta_{1}\delta_{2}-\gamma_{0}b_{2}^{*}}{\mu_{2}};\quad
\beta_{2}:=\dfrac{b_{2}-\gamma_{1}\delta_{2}}{\mu_{2}};\quad
\gamma_{2}:=\dfrac{c_{2}}{\mu_{2}}$}
\State{$z_{0}:=\dfrac{y_{0}}{\mu_{0}};\quad
z_{1}:=\dfrac{y_{1}-\delta_{1}z_{0}}{\mu_{1}};\quad
z_{2}:=\dfrac{y_{2}-\delta_{2}z_{1}-b_{2}^{*}z_{0}}{\mu_{2}}$}
\State{$\xi_{0}:=0.0;\quad\xi_{1}:=0.0;\quad\xi_{2}:=0.0$}
\For{$\overline{i=3,\ldots N-4}$} \Comment{($i$)}
\State{$\delta_{i}:=a_{i}^{*}-\alpha_{i-2}b_{i}^{*}-c_{i}^{*}
(\beta_{i-3}-\alpha_{i-3}\alpha_{i-2})$} \State{$\xi_{i}:=b_{i}^{*}-\alpha_{i-3}c_{i}^{*}$}
\State{$\mu_{i}:=d_{i}-\alpha_{i-1}\delta_{i}-\beta_{i-2}\xi_{i} -\gamma_{i-3}c_{i}^{*}$}
\If{!flag}\If{$|\mu_{i}|<\varepsilon$}\State{$\mu_{i}:=$ symb; flag = True}\EndIf\EndIf
\State{$\alpha_{i}:=\dfrac{a_{i}-\beta_{i-1}\delta_{i}-\gamma_{i-2}\xi_{i}}{\mu_{i}};\quad
\beta_{i}:=\dfrac{b_{i}-\gamma_{i-1}\delta_{i}}{\mu_{i}};\quad
\gamma_{i}:=\dfrac{c_{i}}{\mu_{i}}$}
\State{$z_{i}:=\dfrac{y_{i}-\delta_{i}z_{i-1}-\xi_{i}z_{i-2}-c_{i}^{*}z_{i-3}}{\mu_{i}}$}
\EndFor
\State{$\delta_{N-3}:=a_{N-3}^{*}-\alpha_{N-5}b_{N-3}^{*}-c_{N-3}^{*}
(\beta_{N-6}-\alpha_{N-6}\alpha_{N-5})$} \Comment{($N-3$)}
\State{$\xi_{N-3}:=b_{N-3}^{*}-\alpha_{N-6}c_{N-3}^{*}$}
\State{$\mu_{N-3}:=d_{N-3}-\alpha_{N-4}\delta_{N-3}-\beta_{N-5}\xi_{N-3} -\gamma_{N-6}c_{N-3}^{*}$}
\If{!flag}\If{$|\mu_{N-3}|<\varepsilon$}\State{$\mu_{N-3}:=$ symb; flag = True}\EndIf\EndIf
\State{$\alpha_{N-3}:=\dfrac{a_{N-3}-\beta_{N-4}\delta_{N-3}-\gamma_{N-5}\xi_{N-3}}{\mu_{N-3}}$}
\State{$\beta_{N-3}:=\dfrac{b_{N-3}-\gamma_{N-4}\delta_{N-3}}{\mu_{N-3}}$}
\State{$z_{N-3}:=\dfrac{y_{N-3}-\delta_{N-3}z_{N-4}-\xi_{N-3}z_{N-5}-c_{N-3}^{*}z_{N-6}}{\mu_{N-3}}$}
\State{$\delta_{N-2}:=a_{N-2}^{*}-\alpha_{N-4}b_{N-2}^{*}-c_{N-2}^{*}
(\beta_{N-5}-\alpha_{N-5}\alpha_{N-4})$} \Comment{($N-2$)}
\State{$\xi_{N-2}:=b_{N-2}^{*}-\alpha_{N-5}c_{N-2}^{*}$}
\State{$\mu_{N-2}:=d_{N-2}-\alpha_{N-3}\delta_{N-2}-\beta_{N-4}\xi_{N-2} -\gamma_{N-5}c_{N-2}^{*}$}
\If{!flag}\If{$|\mu_{N-2}|<\varepsilon$}\State{$\mu_{N-2}:=$ symb; flag = True}\EndIf\EndIf
\State{$\alpha_{N-2}:=\dfrac{a_{N-2}-\beta_{N-3}\delta_{N-2}-\gamma_{N-4}\xi_{N-2}}{\mu_{N-2}}$}
\State{$z_{N-2}:=\dfrac{y_{N-2}-\delta_{N-2}z_{N-3}-\xi_{N-2}z_{N-4}-c_{N-2}^{*}z_{N-5}}{\mu_{N-2}}$}
\State{$\delta_{N-1}:=a_{N-1}^{*}-\alpha_{N-3}b_{N-1}^{*}-c_{N-1}^{*}
(\beta_{N-4}-\alpha_{N-4}\alpha_{N-3})$} \Comment{($N-1$)}
\State{$\xi_{N-1}:=b_{N-1}^{*}-\alpha_{N-4}c_{N-1}^{*}$}
\State{$\mu_{N-1}:=d_{N-1}-\alpha_{N-2}\delta_{N-1}-\beta_{N-3}\xi_{N-1} -\gamma_{N-4}c_{N-1}^{*}$}
\If{!flag}\If{$|\mu_{N-1}|<\varepsilon$}\State{$\mu_{N-1}:=$ symb; flag = True}\EndIf\EndIf
\State{$z_{N-1}:=\dfrac{y_{N-1}-\delta_{N-1}z_{N-2}-\xi_{N-1}z_{N-3}-c_{N-1}^{*}z_{N-4}}{\mu_{N-1}}$}
\State{$x_{N-1}:=z_{N-1}$} \Comment{Step 2. Solution}
\State{$x_{N-2}:=z_{N-2}-\alpha_{N-2}x_{N-1}$}
\State{$x_{N-3}:=z_{N-3}-\alpha_{N-3}x_{N-2}-\beta_{N-3}x_{N-1}$}
\For{$\overline{j=N-4,\ldots 0}$}
\State{$x_{j}:=z_{j}-\alpha_{j}x_{j+1}-\beta_{j}x_{j+2}-\gamma_{j}x_{x+3}$}
\EndFor
\State{Cancel the common factors in the numerators and denominators of
$\mathbf{x}$, making them coprime. Substitute $\textrm{symb}:=0$ in
$\mathbf{x}$ and simplify.}
\end{algorithmic}
\HRule

\textbf{Remark:} If any $\mu_{i}$ expression has been evaluated to be zero or
numerically zero, then it is assigned to be a symbolic variable. We cannot
compare any of the next $\mu$ expressions with $\varepsilon$, because any
further $\mu$ is going to be a symbolic expression. To that reason, we use a
boolean flag which tells us if any previous $\mu$ is a symbolic expression. In
that case, comparison with $\varepsilon$ is not conducted as being not needed.


\section{Stability of the Algorithm}

Some observations on the stability of the proposed algorithm can be made.
Firstly, assigning $\mu_{i}, i=\overline{0, N-1}$ to be equal to a symbolic
variable in case it is zero or numerically zero, ensures correctness of the
formulae for computing the solution of the considered SLAE. This action does
not add any additional requirements to the coefficient matrix, except:
\begin{theorem}
The only requirement to the coefficient matrix so as the algorithm to be
stable is nonsingularity.
\end{theorem}
\begin{proof}
As a direct consequence of the transformations done so as the matrix $A$ to be
factorized and then the downwards sweep to be conducted, it follows that the
determinant of the matrix $A$ in the terms of the introduced notation is:
\begin{equation}
\label{eq_1}
\textrm{det}(A) = \prod_{i=0}^{N-1}\mu_{i}|_{\textrm{symb}=0}.
\end{equation}
(This formula could be used so as the nonsingularity of the coefficient matrix
to be checked.)
If $\mu_{i}$ for any $i$ is assigned to be equal to a symbolic variable, then
it is going to appear in both the numerator and the denominator of the
expression for the determinant and so it can be cancelled:
\begin{equation}
\begin{aligned}
\label{eq_2}
\textrm{det}(A)
&= \mu_{0}\,\mu_{1}\,\mu_{2}\ldots\mu_{N-2}\,\mu_{N-1} = \\
&= M_{0}\,\frac{M_{1}}{\mu_{0}}\,\frac{M_{2}}{\mu_{0}\,\mu_{1}}\ldots
\frac{M_{N-2}}{\mu_{0}\,\mu_{1}\ldots\mu_{N-3}}\,\frac{M_{N-1}}
{\mu_{0}\,\mu_{1}\ldots\mu_{N-2}} = \\
& = \frac{\prod_{i=0}^{N-1}M_{i}}{\mu_{0}^{N-1}\,\mu_{1}^{N-2}\,\mu_{2}^{N-3}
\ldots\mu_{N-3}^{2}\,\mu_{N-2}^{1}} = \\
& = \frac{\prod_{i=0}^{N-1}M_{i}}{M_{0}^{N-1}\,\frac{M_{1}^{N-2}}{\mu_{0}^{N-2}}\,
\frac{M_{2}^{N-3}}{\mu_{0}^{N-3}\,\mu_{1}^{N-3}}\ldots\frac{M_{N-3}^{2}}
{\mu_{0}^{2}\,\mu_{1}^{2}\ldots\mu_{N-4}^{2}}\,\frac{M_{N-2}^{1}}{\mu_{0}^{1}\,
\mu_{1}^{1}\ldots\mu_{N-3}^{1}}} = \\
& = \frac{\prod_{i=0}^{N-1}M_{i}}{M_{0}^{N-1}\,\frac{M_{1}^{N-2}}{\mu_{0}^{N-2}}\,
\frac{M_{2}^{N-3}}{\mu_{0}^{N-3}\,\left(\frac{M_{1}}{\mu_{0}}\right)^{N-3}}\ldots
\frac{M_{N-3}^{2}}
{\mu_{0}^{2}\,\left(\frac{M_{1}}{\mu_{0}}\right)^{2}\ldots
\left(\frac{M_{N-4}}{\mu_{N-3}}\right)^{2}}\,\frac{M_{N-2}^{1}}{\mu_{0}^{1}\,
\left(\frac{M_{1}}{\mu_{0}}\right)^{1}\ldots\left(\frac{M_{N-3}}{\mu_{N-4}}\right)^{1}}} = \\
& = \frac{\prod_{i=0}^{N-1}M_{i}}{\prod_{i=0}^{N-2}M_{i}} = M_{N-1}
\end{aligned}
\end{equation}
where $M_{i}$ is the $i$-th leading principal minor, and $\mu_{0} = M_{0}$.
This means that the only constraint on the coefficient matrix is $M_{N-1}\neq0$.
\end{proof}
The requirement on the coefficient matrix to be nonsingular is not limiting at
all since this is a standard requirement so as the SLAE to have one solution
only.


\section*{Acknowledgements}
The work is partially supported by the Russian Foundation 
for Basic Research under project \#18-51-18005.


\section{Conclusions}

A symbolic algorithm for solving band matrix SLAEs with heptadiagonal
coefficient matrices was presented in pseudocode. Some notes on the stability
of the algorithm were made.



\begin{thebibliography}{99.}
\bibitem{Kim} Kim J.G., Park H.W. {\it Advanced Simulation Technique for
    Modeling Multiphase Fluid Flow in Porous Media.} In: Lagan\'{a} A.,
    Gavrilova M.L., Kumar V., Mun Y., Tan C.J.K., Gervasi O. et al.
    Computational Science and Its Applications -- ICCSA 2004. ICCSA 2004.
    Lecture Notes in Computer Science, 2004, \textbf{3044}. Springer, Berlin,
    Heidelberg, pp.\,1 -- 9, \\
    doi: 10.1007/978-3-540-24709-8\_1.
\bibitem{Duran} Duran A., Celebi M.S., Piskin S. et al. {\it Scalability of
    {OpenFOAM} for bio-medical flow simulations.} The Journal of Supercomputing,
    2015, \textbf{71}, pp.\,938 -- 951, \\
    doi: 10.1007/s11227-014-1344-1.
\bibitem{El-Mikkawy} El-Mikkawy, M. {\it A Generalized symbolic Thomas algorithm.}
    Applied Mathematics, 2012, \textbf{3}, 4, pp.\,342 -- 345, doi:
    10.4236/am.2012.34052.
\bibitem{Higham} Higham N.~J. {\it Accuracy and Stability of Numerical
    Algorithms.} SIAM, 2nd edn, 2002, \\
    pp.\,174 -- 176.
\bibitem{Karawia_2013a} Karawia A.~A., Rizvi Q. M. {\it On solving a general
    bordered tridiagonal linear system.} International Journal of Mathematical
    Sciences, 2013, \textbf{33}, 2.
\bibitem{Atlan} Atlan F., El-Mikkawy M. {\it A new symbolic algorithm for
    solving general opposite-bordered tridiagonal linear systems.} American
    Journal of Computational Mathematics, 2015, \textbf{5}, \\
    pp.\,258 -- 266, doi: 10.4236/ajcm.2015.53023.
\bibitem{Askar} Askar S.~S., Karawia A.~A. {\it On solving pentadiagonal linear
    systems via transformations.} Mathematical Problems in Engineering. Hindawi
    Publishing Corporation, 2015 \textbf{2015}, 9, \\
    doi: 10.1155/2015/232456.
\bibitem{Jia}
    Jia J.-T., Jiang Y.-L. {\it Symbolic algorithm for solving cyclic
    penta-diagonal linear systems.} Numerical Algorithms, 2012, \textbf{63},
    2, pp.\,357 -- 367, doi: 10.1007/s11075-012-9626-2.
\bibitem{Veneva_2019}
    Veneva M., Ayriyan A. {\it Performance Analysis of Effective Methods for
    Solving Band Matrix SLAEs after Parabolic Nonlinear PDEs.} Advanced
    Computing in Industrial Mathematics, Revised Selected Papers of the 12th
    Annual Meeting of the Bulgarian Section of SIAM, December 20-22, 2017,
    Sofia, Bulgaria, Studies in Computational Intelligence, vol. 793, Springer
    International Publishing (X), pp.\,407--419, arXiv: 1804.09666v1 [math.NA].
\bibitem{Veneva}
    Veneva M., and Ayriyan A. {\it Performance analysis of effective symbolic
    methods for solving band matrix SLAEs.} Submitted to European Physics
    Journal -- Web of Conferences~(EPJ-WoC).
\bibitem{Veneva_2018}
    Veneva M., Ayriyan A. {\it Effective Methods for Solving Band SLAEs after
    Parabolic Nonlinear PDEs.} AYSS-2017, European Physics Journal -- Web of
    Conferences~(EPJ-WoC). 2018, \textbf{177}, 07004, arXiv: 1710.00428v2
    [math.NA].
\bibitem{Karawia_2013b} Karawia A.~A. {\it A new algorithm for general cyclic
    heptadiagonal linear systems using Sherman-Morrisor-Woodbury formula.}
    {ARS} Combinatoria, 2013, \textbf{108}, pp.\,431--443, \\
    arXiv: 1011.4580 [cs.NA].
\end{thebibliography}
\end{document}